\documentclass[12pt]{article}
\usepackage[a4paper,margin=1in]{geometry}
\usepackage{amssymb}
\usepackage{amsmath}
\usepackage{amsfonts}
\usepackage{amsthm}
\usepackage{color}
\usepackage{float}
\usepackage[all]{xy}
\usepackage{scalerel}
\usepackage{mathabx}
\usepackage{delimset}
\usepackage{esint}

\def\Exp{\mathop{\mbox{\textup{Exp}}}\nolimits}

\newcommand{\fibonomial}{\genfrac{\{}{\}}{0pt}{}}

\newcommand{\C}{\mathbb{C}}

\newcommand{\K}{\mathbb{K}}
\newcommand{\N}{\mathbb{N}}

\newcommand{\R}{\mathbb{R}}
\newcommand{\Per}{\mathbb{P}}
\newcommand{\e}{\mathrm{e}}
\newcommand{\EE}{\mathrm{E}}

\newtheorem{theorem}{Theorem}

\newtheorem{definition}{Definition}
\newtheorem{proposition}{Proposition}
\newtheorem{corollary}{Corollary}
\newtheorem{example}{Example}

\begin{document}

\title{\textbf{First Order Linear Proportional Difference Equation with Integration Factor the $(s,t)$-Pantograph Function}}
\author{Ronald Orozco L\'opez}
\newcommand{\Addresses}{{% additional braces for segregating \footnotesize
  \bigskip
  \footnotesize

  %R.~Orozco, \textsc{Departamento de Matematicas, Universidad de los Andes, 
  % Carrera 1 N. 18A-12 Bogot\'a, Colombia}\par\nopagebreak
  \textit{E-mail address}, R.~Orozco: \texttt{rj.orozco@uniandes.edu.co}
  
}}

\maketitle

\tableofcontents

\begin{abstract}
In this paper we find solutions to first-order linear proportional difference equations via the $(s,t)$-integration factor method. The $(s,t)$-integration factor involves the $(s,t)$-Pantograph function, which is a generalization of the partial Theta function. Other equations are solved including the $(s,t)$-analogue of the Bernoulli equation.
\end{abstract}


\section{Introduction}

Orozco \cite{orozco} studied the existence of equations in differences with proportional delay of the form
\begin{equation}
    y(ax)-y(bx)=(a-b)x f(x,y(x),y(ux)),\ y(\eta)=\xi,
\end{equation}
where $a,b,u\in\R$. In this paper, we will be mainly interested in solving equations in proportional differences of the form
\begin{equation}\label{eqn_foplde}
    y(ax)-y(bx)=(a-b)x(\beta(x)-\alpha(x)y(u x)),
\end{equation}
or equivalently
\begin{equation*}
    \mathbf{D}_{a,b}y(x)+\alpha(x)y(ux)=\beta(x),
\end{equation*}
by means of the integration factor method involving the $(s,t)$-Pantograph function
\begin{equation*}
    \EE_{s,t}(a,b;x,u)=\sum_{n=0}^{\infty}(a\oplus b)_{1,u}^{(n)}\frac{x^{n}}{\brk[c]{n}_{s,t}!}
\end{equation*}
which is a solution of the equation
\begin{equation}\label{eqn_panto_st}
    \mathbf{D}_{s,t}y(x)=ay(x)+by(ux).
\end{equation}
The special case $s=2$ and $t=-1$ is studied in \cite{ise_1,ise_2}.
The function $\EE_{s,t}(a,b;x,u)$ generalises the $(s,t)$-exponential and the partial Theta function $\Theta_{0}(x,y)$ of Ramanujan. Some solutions will then be expressed in terms of partial Theta functions. Since the integration factor $I=\exp(\int_{\eta}^{x}\alpha(t)dt)$ for ordinary differential equations involves the composition of functions, then it will be necessary to introduce a type of composition analogous to the $q$-composition of Gessel \cite{gessel} and Johnson \cite{john}.

In addition, we shall give a power series solution to the equation
\begin{equation*}
    y(ax)-y(bx)=(a-b)x(\alpha(cy(x)+dy(ux))+\beta(x)),
\end{equation*}
and we will provide a solution via operators for the equation
\begin{equation*}
    y(ax)-y(bx)=(a-b)x(c\beta y(x)+\gamma y(ux)+\delta\EE_{s,t}(c,d;\alpha x,u)).
\end{equation*}
Finally, we will consider the deformed $(s,t)$-Bernoulli equation
\begin{multline*}
    D_{\varphi^{n-1},\varphi^{\prime(n-1)}}y(x)+\alpha(x)\frac{(D_{q}E_{s,t})(a,b;x,u)\square A(x)}{\EE_{s,t}[a,b;A(\varphi_{s,t}x),u]}y(\varphi_{s,t}^{n-1}x)
    =\beta(x)\prod_{j=0}^{\infty}\frac{y(q^{j}\varphi_{s,t}^{n-1}x)}{y(q^{n+j}\varphi_{s,t}^{n-1}x)},
\end{multline*}
where
\begin{equation*}
    D_{\varphi^{n-1},\varphi^{\prime(n-1)}}f(x)=\frac{f(\varphi^{n-1}x)-f(\varphi^{\prime(n-1)})}{(\varphi^{n-1}-\varphi^{\prime(n-1)})x}
\end{equation*}
is a $(\varphi^{n-1},\varphi^{\prime(n-1)})$-derivative.

We divide this paper as follows. Section 1 is a short introduction. Section 2 deals with the results and formulas of the $q$-calculus and the calculus on generalized Fibonacci polynomials needed to develop this paper. Section 3 introduces the $(s,t)$-composition, the deformed $(s,t)$-chain rule, and its properties are given. The final section discusses the main objective of this paper, which is to solve first-order linear proportional difference equations.

\section{Preliminaries}

\subsection{Notation for $q$-calculus}
The $q$-shifted factorial is
\begin{equation}
    (a;q)_{n}=
    \begin{cases}
    1,&n=0;\\
    \prod_{k=0}^{n-1}(1-aq^{k}),&n\geq1.
    \end{cases}
\end{equation}
Also, define 
\begin{equation}
    (a;q)_{\infty}=\lim_{n\rightarrow\infty}(a;q)_{n}=\prod_{k=0}^{\infty}(1-aq^{k}).
\end{equation}
For a complex number $a$ and $q\neq1$ the $q$-number $[a]_{q}$ is defined by
\begin{equation}
    [a]_{q}=\frac{1-q^a}{1-q}.
\end{equation}
In terms of $q$-numbers the $q$-number factorial $[n]_{q}!$ is defined for a nonnegative integer $n$ by $[n]_{q}!=\prod_{k=1}^{n}[k]_{q}$. Another expression for $[n]_{q}!$ is
\begin{equation}\label{eqn_qfac}
    [n]_{q}!=\prod_{k=0}^{n-1}\frac{1-q^{k+1}}{1-q}=\frac{(q;q)_{n}}{(1-q)^n}.
\end{equation}
The ${}_r\phi_{s}$ basic hypergeometric series is define by
\begin{equation}
    {}_r\phi_{s}\left(
    \begin{array}{c}
         a_{1},a_{2},\ldots,a_{r} \\
         b_{1},\ldots,b_{s}
    \end{array}
    ;q,z
    \right)=\sum_{n=0}^{\infty}\frac{(a_{1};q)_{n}(a_{2};q)_{n}\cdots(a_{r};q)_{n}}{(q;q)_{n}(b_{1};q)_{n}\cdots(b_{s};q)_{n}}\Big[(-1)^{n}q^{\binom{n}{2}}\Big]^{1+s-r}z^n.
\end{equation}
The $q$-binomial theorem states that
\begin{equation}
    {}_1\phi_{0}(a;q,z)=\frac{(az;q)_{\infty}}{(z;q)_{\infty}}=\sum_{n=0}^{\infty}\frac{(a;q)_{n}}{(q;q)_{n}}z^{n}.
\end{equation}
The $q$-exponential $\e_{q}(z)$ is defined by
\begin{align}
    \e_{q}(z)=\sum_{n=0}^{\infty}\frac{z^n}{[n]_{q}!}&=\sum_{n=0}^{\infty}\frac{((1-q)z)^n}{(q;q)_{n}}\nonumber\\
    &={}_1\phi_{0}\left(\begin{array}{c}
         0\\
         - 
    \end{array};q,-(1-q)z\right)=\frac{1}{((1-q)z;q)_{\infty}}.
\end{align}
Another $q$-analogue of the classical exponential function is
\begin{align}
    \EE_{q}(z)=\sum_{n=0}^{\infty}q^{\binom{n}{2}}\frac{z^n}{[n]_{q}!}&=\sum_{n=0}^{\infty}q^{\binom{n}{2}}\frac{((1-q)z)^n}{(q;q)_{n}}\nonumber\\
    &={}_1\phi_{1}\left(\begin{array}{c}
         0\\
         0 
    \end{array};q,-(1-q)z\right)=(-(1-q)z;q)_{\infty}
\end{align}

\subsection{Generalized Fibonacci polynomials}

The generalized Fibonacci polynomials depending on the variables $s,t$ are defined by
\begin{align*}
    \brk[c]{0}_{s,t}&=0,\\
    \brk[c]{1}_{s,t}&=1,\\
    \brk[c]{n+2}_{s,t}&=s\{n+1\}_{s,t}+t\{n\}_{s,t}.
\end{align*}
As special cases, we obtain the sequences of Pell, Jacobsthal, Chebysheff of the second kind, Mersenne, and Repunits, among others. When $s=P$ and $t=-Q$, with $P$ and $Q$ integers, we obtain the $(P,-Q)$-Lucas differential calculus and if $Q=-1$, we obtain the $P$-Fibonacci differential calculus.
The $(s,t)$-Fibonacci constant is the ratio toward which adjacent $(s,t)$-Fibonacci polynomials tend. This is the only positive root of $x^{2}-sx-t=0$. We will let $\varphi_{s,t}$ denote this constant, where
\begin{equation*}
    \varphi_{s,t}=\frac{s+\sqrt{s^{2}+4t}}{2}
\end{equation*}
and
\begin{equation*}
    \varphi_{s,t}^{\prime}=s-\varphi_{s,t}=-\frac{t}{\varphi_{s,t}}=\frac{s-\sqrt{s^{2}+4t}}{2}.
\end{equation*}
The $(s,t)$-Fibonomial coefficients are define by
\begin{equation*}
    \fibonomial{n}{k}_{s,t}=\frac{\brk[c]{n}_{s,t}!}{\brk[c]{k}_{s,t}!\brk[c]{n-k}_{s,t}!},
\end{equation*}
where $\brk[c]{n}_{s,t}!=\brk[c]{1}_{s,t}\brk[c]{2}_{s,t}\cdots\brk[c]{n}_{s,t}$ is the $(s,t)$-factorial or generalized fibotorial. From Eq.(\ref{eqn_qfac}),
\begin{align*}
    \brk[c]{n}_{s,t}!&=\varphi_{s,t}^{\binom{n}{2}}\frac{(q;q)_{n}}{(1-q)^n}.
\end{align*}

\subsection{Calculus on generalized Fibonacci polynomials}

For all $u\in\C$ and for all $s,t\in\R$ such that $s\neq0$, $t\neq0$, define the $(s,t)$-derivative $\mathbf{D}_{s,t}$ of a function $f(x)$ as
\begin{equation*}
(\mathbf{D}_{s,t}f)(x)=\frac{f(\varphi_{s,t}x)-f(\varphi_{s,t}^{\prime}x)}{(\varphi_{s,t}-\varphi_{s,t}^{\prime})x},
\end{equation*}
for all $x\neq0$ and $(\mathbf{D}_{s,t}f)(0)=f^{\prime}(0)$, provided $f^{\prime}(0)$ exists. The equivalence that exists between the $q$-derivative and the $(s,t)$-derivative is given by the following identity:
\begin{equation*}
    (\mathbf{D}_{s,t}f)(x)=(D_{q}f)(\varphi_{s,t}x),
\end{equation*}
where 
\begin{equation*}
    D_{q}f(x)=\frac{f(x)-f(qx)}{(1-q)x}
\end{equation*}
and $q=\varphi_{s,t}^{\prime}/\varphi_{s,t}$. We will say that a function $f(x)$ is $q$-periodic if $f(y)=f(qy)$, with $y=\varphi_{s,t}x$. Let $\Per_{s,t}$ denote the set of  $q$-periodic functions. The $q$-periodic functions form the kernel of the operator $\mathbf{D}_{s,t}$. Set $s\neq0$ and $t<0$. From the condition of $q$-periodicity of $f(x)$ it follows that $f(y)=f(qy)$, with $q>0$. Then
\begin{align*}
    f(q^y)&=f(q^{y+1})\\
    G(y)&=G(y+1)
\end{align*}
where $G$ is an arbitrary periodic function with period one and $f(x)=G(\log_{q}(x))$, $x>0$. If $t>0$, then $q<0$ and $f(x)=G(\log(x)/(\log(-q)+i\pi))$, so that $x\in\C/\{0\}$ and $f(x)\in\Per_{s,t}$.

Let $f$ be an arbitrary function. In \cite{nji} the following $(p,q)$-integrals of $f$ are defined
\begin{align*}
    \int_{0}^{a} f(x)d_{p,q}x&=(p-q)a\sum_{k=0}^{\infty}\frac{q^{k}}{p^{k+1}}f\left(\frac{q^{k}}{p^{k+1}}a\right)\text{ if  }\ \Bigg\vert\frac{p}{q}\Bigg\vert>1,\\
    \int_{0}^{a}f(x)d_{p,q}x&=(q-p)a\sum_{k=0}^{\infty}\frac{p^{k}}{q^{k+1}}f\left(\frac{p^{k}}{q^{k+1}}a\right)\text{ if  }\ \Bigg\vert\frac{p}{q}\Bigg\vert<1
\end{align*}

\begin{align*}
    \int_{a}^{b}f(x)d_{p,q}x&=\int_{0}^{b}f(x)d_{p,q}x-\int_{0}^{a}f(x)d_{p,q}x
\end{align*}

For $\alpha\in\R$, the set $\alpha I\subset\R$ is defined by $\alpha I=\{\alpha x\ :\ x\in I\}$. Set $\vert q\vert$ such that $0<\vert q\vert<1$. We will say that $I$ is a $(q,\varphi)$-geometric set if it contains all geometric sequences $\{xq^n/\varphi_{s,t}\}$, $n\in\N_{0}=\N\cup\{0\}$, $x\in I$. 
For a pair of real numbers $a$ and $b$, we define
\begin{equation*}
    [a,b]_{q}=\Bigg\{a\frac{q^n}{\varphi_{s,t}}\ :\ n\in\N_{0}\Bigg\}\cup\Bigg\{b\frac{q^n}{\varphi_{s,t}}\ :\ n\in\N_{0}\Bigg\}
\end{equation*}
which we will call a $(q,\varphi)$-interval with extreme points $a$ and $b$. For real numbers $a$ and $b$, the following property holds:
\begin{equation*}
    \text{If }a,b\in I\ \Longrightarrow\ [a,b]_{q}\subset I.
\end{equation*}
Consider the function $f:[a,b]_{q}\rightarrow\K$, where $\K$ is either $\R$ or $\C$. The $(s,t)$-integral in $[a,b]_{q}$ of $f$ is 
\begin{equation}\label{eqn_int_def}
    \int_{a}^{b}f(x)d_{s,t}x=(1-q)\sum_{n=0}^{\infty}\Bigg[bf\left(b\frac{q^n}{\varphi_{s,t}}\right)-af\left(a\frac{q^n}{\varphi_{s,t}}\right)\Bigg]q^n=\int_{a}^{b}f(x/\varphi_{s,t})d_{q}x.
\end{equation}
We will call $f$ a function $(s,t)$-integrable if the series in Eq.(\ref{eqn_int_def}) is convergent. 

For any real number $p\geq1$ and $a,b\in I$, we will denote by $\mathcal{L}_{s,t}^{p}[a,b]$ the set of functions $f:[a,b]_{q}\rightarrow\K$ such that $\vert f\vert^{p}$ is $(s,t)$-integrable in $[a,b]_{q}$, i.e.,
\begin{equation*}
    \mathcal{L}_{s,t}^{p}[a,b]=\Bigg\{f:[a,b]_{q}\rightarrow\K\ :\ \int_{[a,b]_{q}}\vert f\vert^pd_{s,t}<\infty\Bigg\}.
\end{equation*}
We also set
\begin{equation*}
    \mathcal{L}_{s,t}^{\infty}[a,b]=\Bigg\{f:[a,b]_{q}\rightarrow\K\ :\ \sup_{n\in\N_{0}}\Bigg\{\Bigg\vert f\left(\frac{aq^n}{\varphi_{s,t}}\right)\Bigg\vert,\Bigg\vert f\left(\frac{bq^n}{\varphi_{s,t}}\right)\Bigg\vert\Bigg\}<\infty\Bigg\}
\end{equation*}

{\bf The fundamental theorem of $(s,t)$-calculus}.
Let $0<\vert q\vert<1$, and $I=[a,b]_{q}$ a $(q,\varphi)$-interval. Fix $a,b\in I$, and let $f:I\rightarrow\K$ be a function such that $\mathbf{D}_{s,t}f\in\mathcal{L}_{s,t}^{1}[a,b]$. Then:
if $f$ is continuous at $0$, then
    \begin{equation}\label{theo_funda}
        \int_{a}^{b}(\mathbf{D}_{s,t}f)(x)d_{s,t}x=f(b)-f(a).
    \end{equation}

{\bf The $(s,t)$-integration by parts formula}. Assume $0<\vert q\vert<1$, and let $I=[a,b]_{q}$ be a $(q,\varphi)$-interval. Fix $a,b\in I$, and let $f:I\rightarrow\K$ be a function such that $\mathbf{D}_{s,t}f\in\mathcal{L}_{s,t}^{1}[a,b]$. Then the equality
\begin{equation}\label{theo_partes2}
    \int_{I}(\mathbf{D}_{s,t}f)(x)g(\varphi_{s,t}^{\prime}x)d_{s,t}x=\Big[f\cdot g\Big]_{a}^{b}
    -\int_{I}f(\varphi_{s,t}x)(\mathbf{D}_{s,t}g)(x)d_{s,t}x
\end{equation}
holds, provided $f,g\in\mathcal{L}_{s,t}^{1}[a,b]$, $(\mathbf{D}_{s,t}f)(x)$ and $(\mathbf{D}_{s,t}g)(x)$ are bounded for all $x\in I$, and the limits exists. If, in addition, $f$ and $g$ are continuous at $0$, then

\subsection{Deformed $(s,t)$-exponential function}

Set $s\neq0$, $t\neq0$, $s,t\in\R$. For all $u\in\C$, the deformed $(s,t)$-exponential function 
\begin{equation*}
    \exp_{s,t}(z,u)=
    \begin{cases}
        \sum_{n=0}^{\infty}u^{\binom{n}{2}}\frac{z^{n}}{\brk[c]{n}_{s,t}!}&\text{ if }u\neq0;\\
        1+z&\text{ if }u=0,
    \end{cases}
\end{equation*}
is solution of the equation $\mathbf{D}_{s,t}y=y(ux)$. The $(s,t)$-antiderivative of $\exp_{s,t}(x,u)$ is
\begin{equation}
    \int\exp_{s,t}(x,u)d_{s,t}x=u\exp_{s,t}(x/u,u).
\end{equation}
Also, we define
\begin{align*}
    \Exp_{s,t}(x)&=\exp_{s,t}(x,\varphi_{s,t})\\
    \Exp_{s,t}^{\prime}(x)&=\exp_{s,t}(x,\varphi_{s,t}^\prime)
\end{align*}
and
\begin{equation}
    \exp_{s,t}((\alpha\oplus\beta)_{\varphi,\varphi^\prime}x)=\sum_{n=0}^{\infty}(\alpha\oplus\beta)_{\varphi,\varphi^\prime}^{n}\frac{x^n}{\brk[c]{n}_{s,t}!}=\Exp_{s,t}(x)\Exp_{s,t}^{\prime}(x)
\end{equation}
where
\begin{equation}\label{eqn_binom}
    (\alpha\oplus\beta)_{\varphi,\varphi^\prime}^n=
    \begin{cases}
        \prod_{k=0}^{n-1}(\alpha\varphi_{s,t}^k+\beta\varphi_{s,t}^{\prime k}),&\text{ if }n\geq1;\\
        1,&\text{ if }n=0.
    \end{cases}
\end{equation}

\begin{proposition}\label{prop_exp_prod}
For all $\alpha,\beta\in\C$. Then
    \begin{enumerate}
        \item $\mathbf{D}_{s,t}\exp_{s,t}((\alpha\oplus\beta)_{\varphi,\varphi^\prime}x)=\alpha\exp_{s,t}((\alpha\oplus\beta)_{\varphi,\varphi^\prime}\varphi_{s,t}x)+\beta\exp_{s,t}((\alpha\oplus\beta)_{\varphi,\varphi^\prime}\varphi_{s,t}^{\prime}x)$.
        \item $\int\exp_{s,t}((\alpha\oplus\beta)_{\varphi,\varphi^\prime}x)d_{s,t}x=\frac{\varphi_{s,t}\varphi_{s,t}^\prime}{\alpha\varphi_{s,t}^\prime+\beta\varphi_{s,t}}\exp_{s,t}(((\alpha/\varphi_{s,t})\oplus(\beta/\varphi_{s,t}^\prime))_{\varphi,\varphi^\prime}x)$
    \end{enumerate}
\end{proposition}
\begin{proof}
From Eq.(\ref{eqn_binom}),
\begin{align*}
    \mathbf{D}_{s,t}\exp_{s,t}((\alpha\oplus\beta)_{\varphi,\varphi^{\prime}}x)&=\mathbf{D}_{s,t}\left(\sum_{n=0}^{\infty}(\alpha\oplus\beta)_{\varphi,\varphi^\prime}^{n}\frac{x^n}{\brk[c]{n}_{s,t}!}\right)\\
    &=\sum_{n=0}^{\infty}(\alpha\oplus\beta)_{\varphi,\varphi^\prime}^{n+1}\frac{x^n}{\brk[c]{n}_{s,t}!}\\
    &=\alpha\sum_{n=0}^{\infty}(\alpha\oplus\beta)^{n}_{\varphi,\varphi^\prime}\frac{(\varphi_{s,t}x)^n}{\brk[c]{n}_{s,t}!}+\beta\sum_{n=0}^{\infty}(\alpha\oplus\beta)_{\varphi,\varphi^\prime}^{n}\frac{(\varphi_{s,t}^{\prime}x)^n}{\brk[c]{n}_{s,t}!}\\
    &=\alpha\exp_{s,t}((\alpha\oplus\beta)_{\varphi,\varphi^\prime}\varphi_{s,t}x)+\beta\exp_{s,t}((\alpha\oplus\beta)_{\varphi,\varphi^\prime}\varphi_{s,t}^{\prime}x).
\end{align*}
From $(s,t)$-integration by parts formula Eq.(\ref{theo_partes2}),
\begin{align*}
    \int\exp_{s,t}((\alpha\oplus\beta)_{\varphi,\varphi^{\prime}}x)d_{s,t}x&=\int\Exp_{s,t}(\alpha x)\Exp_{s,t}^{\prime}(\beta x)d_{s,t}x\\
    &=\frac{\varphi_{s,t}}{\alpha}\int\Exp_{s,t}^{\prime}(\beta(\varphi_{s,t}^{\prime}x)/\varphi_{s,t})\mathbf{D}_{s,t}\Exp_{s,t}(\alpha x/\varphi_{s,t})d_{s,t}x\\
    &=\frac{\varphi_{s,t}}{\alpha}\Big[\Exp_{s,t}(\beta x/\varphi_{s,t}^\prime)\Exp_{s,t}(\alpha x/\varphi_{s,t})\\
    &\hspace{2cm}-\frac{\beta}{\varphi_{s,t}^{\prime}}\int\Exp_{s,t}(\alpha x)\Exp_{s,t}^{\prime}(\beta x)d_{s,t}x\Big].
\end{align*} 
Then
\begin{align*}
    \left(1+\frac{\varphi_{s,t}\beta}{\varphi_{s,t}^{\prime}\alpha}\right)\int\Exp_{s,t}(\alpha x)\Exp_{s,t}^{\prime}(\beta x)d_{s,t}x=\frac{\varphi_{s,t}}{\alpha}\Exp_{s,t}^{\prime}(\beta x/\varphi_{s,t}^\prime)\Exp_{s,t}(\alpha x/\varphi_{s,t})
\end{align*}
and
\begin{align*}
    \int\Exp_{s,t}(\alpha x)\Exp_{s,t}^{\prime}(\beta x)d_{s,t}x&=\frac{\varphi_{s,t}\varphi_{s,t}^{\prime}}{\alpha\varphi_{s,t}^\prime+\beta\varphi_{s,t}}\Exp_{s,t}(\alpha x/\varphi_{s,t})\Exp_{s,t}^{\prime}(\beta x/\varphi_{s,t}^\prime)\\
    &=\frac{\varphi_{s,t}\varphi_{s,t}^{\prime}}{\alpha\varphi_{s,t}^\prime+\beta\varphi_{s,t}}\exp_{s,t}(((\alpha/\varphi_{s,t})\oplus(\beta/\varphi_{s,t}^\prime))_{\varphi,\varphi^\prime}x).
\end{align*}
The proof is completed.
\end{proof}

\subsection{The $(s,t)$-pantograph function}

The $(1,u)$-deformed $(s,t)$-pantograph function is defined by
    \begin{equation}\label{eqn_panto}
        \EE_{s,t}(a,b;z,u       
        )=\sum_{n=0}^{\infty}(a\oplus b)_{1,u}^n\frac{z^n}{\brk[c]{n}_{s,t}!},
    \end{equation}
and this is solution of the equation $\mathbf{D}_{s,t}y=ay(x)+by(ux)$, where
\begin{equation}
    (a\oplus b)_{1,u}^{n}=
    \begin{cases}
        \prod_{k=0}^{n-1}(a+bu^k),&\text{ if }n\neq0;\\
        1,&\text{ if }n=0.
    \end{cases}
\end{equation}
For $a\neq0$, it is very easy to notice that $(a\oplus b)_{1,u}^{n}=a^{\binom{n}{2}}(b/a;u)_{n}$. Then
\begin{equation}
    \EE_{s,t}(a,b;x,u)=\sum_{n=0}^{\infty}\left(\frac{a}{\varphi_{s,t}}\right)^{\binom{n}{2}}\frac{(b/a;u)_{n}}{(q;q)_{n}}((1-q)x)^n
\end{equation}
is a $q$-series. Some special values of the function $\EE_{s,t}(a,b,u;x)$ are
\begin{align*}
    \EE_{s,t}(a,-a;x,u)&=0,\\
    \EE_{s,t}(a,0;x,-)&=\exp_{s,t}(ax),\\
    \EE_{s,t}(0,a;x,u)&=\exp_{s,t}(ax,u),\\
    \EE_{s,t}(a,b;x,1)&=\exp_{s,t}((a+b)x),\\
    \EE_{s,t}(ac,bc;x,u)&=\EE_{s,t}(a,b;cx,u),\\
    \EE_{s,t}(\varphi_{s,t},-b;-x,q)&={}_1\phi_{0}\left(\begin{array}{c}
      b/\varphi_{s,t}\\
      -   
    \end{array};q,(1-q)x\right)=\frac{((b/\varphi_{s,t})(1-q)x;q)_{\infty}}{((1-q)x;q)_{\infty}},\\
    \EE_{s,t}(1,-q;x,q)&=\Theta_{0}((1-q)x,\varphi_{s,t}^{-1}),
\end{align*}
where $q=\varphi_{s,t}^{\prime}/\varphi_{s,t}$ and
\begin{equation}
    \Theta_{0}(x,y)=\sum_{n=0}^{\infty}y^{\binom{n}{2}}x^{n}
\end{equation}
is the Partial Theta function [?]. From here, is defined the function
\begin{equation*}
    \psi(q)=\Theta_{0}(q,q)=\sum_{n=0}^{\infty}q^{\binom{n+1}{2}}=\frac{(q^2;q^2)_{\infty}}{(q;q^2)_{\infty}}.
\end{equation*}
Also,
\begin{equation}
    \EE_{s,t}(\varphi_{s,t}^{\prime},-b;-x,q)=\sum_{n=0}^{\infty}(-1)^{n}q^{\binom{n}{2}}\frac{(b/\varphi_{s,t}^\prime;q)_{n}}{(q;q)_{n}}(1-q)^{n}x^n,
\end{equation}
and
\begin{equation}
    \EE_{s,t}(\varphi_{s,t}^{\prime},-\varphi_{s,t}^{\prime}q;-x,q)=\sum_{n=0}^{\infty}(-1)^{n}q^{\binom{n}{2}}(1-q)^{n}x^n,
\end{equation}

\begin{theorem}
Set $a,b,u\in\C$ such that $a\neq0$ and $\vert b/au\vert<1$. Then
\begin{equation}
    \int \EE_{s,t}(a,b;x,u)d_{s,t}x=\frac{1}{a}\sum_{k=0}^{\infty}(-1)^{k}\left(\frac{b}{au}\right)^{k}\EE_{s,t}(a,b;u^{k}x,u).
\end{equation}
\end{theorem}
\begin{proof}
From Eq.(\ref{eqn_panto})
\begin{align*}
    \int\EE_{s,t}(a,b;x,u)d_{s,t}x&=\sum_{n=0}^{\infty}(a\oplus b)_{1,u}^{n}\frac{x^{n+1}}{\brk[c]{n+1}_{s,t}!}\\
    &=\sum_{n=1}^{\infty}\frac{(a\oplus b)_{1,u}^n}{a+bu^{n-1}}\frac{x^n}{\brk[c]{n}_{s,t}!}\\
    &=\frac{1}{a}\sum_{n=1}^{\infty}(a\oplus b)_{1,u}^{n}\frac{x^n}{\brk[c]{n}_{s,t}!}\sum_{k=0}^{\infty}(-1)^{k}\left(\frac{b}{au}\right)^{k}u^{kn}\\
    &=\frac{1}{a}\sum_{k=0}^{\infty}(-1)^{k}\left(\frac{b}{au}\right)^{k}\sum_{n=1}^{\infty}(a\oplus b)_{1,u}^{n}\frac{(u^{k}x)^n}{\brk[c]{n}_{s,t}!}\\
    &=\frac{1}{a}\sum_{k=0}^{\infty}(-1)^{k}\left(\frac{b}{au}\right)^k(\EE_{s,t}(a,b;u^{k}x,u)-1)\\
    &=\frac{1}{a}\sum_{k=0}^{\infty}(-1)^{k}\left(\frac{b}{au}\right)^k\EE_{s,t}(a,b;u^{k}x,u)-\frac{u}{au+b}.
\end{align*}
As $\mathbf{D}_{s,t}(\frac{u}{au+b})=0$, then $\frac{1}{a}\sum_{k=0}^{\infty}(-1)^{k}\left(\frac{b}{au}\right)^k\EE_{s,t}(a,b;u^{k}x,u)$ is the $(s,t)$-antiderivative of $\EE_{s,t}(a,b;x,u)$.
\end{proof}

\begin{corollary}\label{cor_int_th}
If $\vert q\vert<1$, then
\begin{equation}
    \int\Theta_{0}((1-q))x,\varphi_{s,t}^{-1})d_{s,t}x=\sum_{k=0}^{\infty}(\Theta_{0}((1-q))q^{k}x,\varphi_{s,t}^{-1})-1).
\end{equation}
\end{corollary}

\section{Deformed $(s,t)$-chain rule}

This section is based on Gessel's paper \cite{gessel}, slightly modified by Johnson \cite{john}, on $q$-function composition and $q$-chain rule. We will define the deformed $(s,t)$-composition of functions and give the $(s,t)$-chain rule.

Consider a function $f(x)$ of the form
\begin{equation*}
    f(x)=\sum_{n=1}^{\infty}f_{n}\frac{x^n}{\brk[s]{n}_{q}!}.
\end{equation*}
Then for $k\geq0$, the $k$th symbolic power $f^{[k]}(x)$ of $f$ is defined by $f^{[0]}(x)=1$ and for $k\geq1$ define the $k$th symbolic power of $f$ inductively by
\begin{equation*}
    D_{q}f^{[k]}(x)=\brk[s]{k}_{q}f^{[k-1]}(x)D_{q}f(x),\ \text{ with $f^{[k]}(0)=0$}.
\end{equation*}
When $k=1$, $\mathbf{D}_{q}f^{[1]}(x)=\mathbf{D}_{q}f(x)$. Then $f^{[1]}(x)=f(x)$ because their derivatives coincide and because both are zero at $x=0$. Note also that $\mathbf{D}_{q}x^{[k]}=\brk[s]{k}_{q}x^{[k-1]}$ and by induction it is shown that $x^{[k]}=x^{k}$, because $x^{[1]}=x$. 
If $r$ does not depend on $x$, then
\begin{equation}\label{eqn_comp_cons}
    (tf(x))^{[k]}=t^{k}f^{[k]}(x).
\end{equation}
For exchange $D_{q}$ by $\mathbf{D}_{s,t}$ in the definition of Gessel of $q$-chain rule, we obtain the $(s,t)$-chain rule of $f^{[n]}(x)$
\begin{equation}\label{eqn_der_pow}
    \mathbf{D}_{s,t}f^{[n]}(x)=\brk[c]{n}_{s,t}f^{[n-1]}(x)\mathbf{D}_{s,t}f(x),\hspace{0.5cm}\text{with }f^{[k]}(0)=0
\end{equation}
when $f(x)$ is defined by
\begin{equation*}
    f(x)=\sum_{n=1}^{\infty}f_{n}\frac{x^n}{\brk[c]{n}_{s,t}!}.
\end{equation*}
Next, we will define the deformed $(s,t)$-composition of functions.

\begin{definition}
Suppose a function $g_{s,t}(x,u)$ defined by
\begin{equation*}
    g_{s,t}(x,u)=\sum_{n=0}^{\infty}u^{\binom{n}{2}}g_{n}\frac{x^n}{\brk[c]{n}_{s,t}!}
\end{equation*}
Then the $u$-deformed $(s,t)$-composition $g_{s,t}[f,u]$ is defined to be
\begin{equation}\label{eqn_comp}
    g_{s,t}(x,u)\square f(x)=g_{s,t}[f,u]=\sum_{n=0}^{\infty}u^{\binom{n}{2}}g_{n}\frac{f^{[n]}(x)}{\brk[c]{n}_{s,t}!}.
\end{equation}
\end{definition}
More generally,
\begin{definition}
Suppose a function $g_{s,t}(a,b;x,u)$ defined by
\begin{equation*}
g_{s,t}(a,b;x,u)=\sum_{n=0}^{\infty}(a\oplus b)_{1,u}^{n}g_{n}\frac{x^{n}}{\brk[c]{n}_{s,t}!}
\end{equation*}
Then the $(1,u)$-deformed $(s,t)$-composition $g_{s,t}[a,b;f,u]$ is defined to be
\begin{equation}\label{eqn_abcomp}
    g_{s,t}(a,b;x,u)\square f(x)=g_{s,t}[a,b;f(x),u]=\sum_{n=0}^{\infty}(a\oplus b)_{1,u}^{n}g_{n}\frac{f^{[n]}(x)}{\brk[c]{n}_{s,t}!}.
\end{equation}
When $a=0$ and $b=1$, then
\begin{equation*}
    g_{s,t}[0,1;f(x),u]=g_{s,t}[f(x),u].
\end{equation*}
\end{definition}
Then it is easy to obtain the $(s,t)$-chain rule for $(1,u)$-deformed functions.
\begin{proposition}
Set $a,b\in\R$ such that $a\neq-b$. Then
\begin{align}
    \mathbf{D}_{s,t}(g_{s,t}[a,b;f(x),u])&=\Big(a(\mathbf{D}_{s,t}g_{s,t})[a,b;f(x),u]\nonumber\\
    &\hspace{2cm}+b(\mathbf{D}_{s,t}g_{s,t})[a,b;ug(x),u]\Big)(\mathbf{D}_{s,t}f)(x).
\end{align}
\end{proposition}
\begin{proof}
From Eq.(\ref{eqn_der_pow}), we have that
    \begin{align*}
    \mathbf{D}_{s,t}(g_{s,t}[a,b;f(x),u])&=\mathbf{D}_{s,t}\left(\sum_{n=0}^{\infty}(a\oplus b)_{1,u}^{n}g_{n}\frac{f^{[n]}(x)}{\brk[c]{n}_{s,t}!}\right)\\
    &=\sum_{n=1}^{\infty}(a\oplus b)_{1,u}^{n}g_{n}\frac{\brk[c]{n}_{s,t}f^{[n-1]}(x)}{\brk[c]{n}_{s,t}!}\mathbf{D}_{s,t}f(x)\\
    &=\sum_{n=0}^{\infty}(a\oplus b)_{1,u}^{n+1}g_{n+1}\frac{f^{[n]}(x)}{\brk[c]{n}_{s,t}!}\mathbf{D}_{s,t}f(x)\\
    &=\Big(a(\mathbf{D}_{s,t}g_{s,t})[a,b;f(x),u]\\
    &\hspace{2cm}+b(\mathbf{D}_{s,t}g_{s,t})[a,b;uf(x),u]\Big)\mathbf{D}_{s,t}f(x)
\end{align*}
and the proof is reached.
\end{proof}
Immediately, we have the following Corollary.
\begin{corollary}\label{coro_chain_rule}
$\mathbf{D}_{s,t}g_{s,t}[f(x),u]=(\mathbf{D}_{s,t}g_{s,t})\big[uf(x),u\big]\mathbf{D}_{s,t}f(x)$.
\end{corollary}
\begin{proof}
Set $a=0$ and $b=1$ in the above Proposition.
\end{proof}

\begin{definition}\label{def_sub}
For $n\in\N$, define
    \begin{equation}
        \sqint_{f(r)}^{g(r)}x^{n}d_{s,t}x\equiv\left(\int x^{n}d_{s,t}x\right)\Bigg\vert_{f(r)}^{g(r)}=\frac{x^{n+1}\square g(r)-x^{n+1}\square f(r)}{\brk[c]{n+1}_{s,t}}.
    \end{equation}
\end{definition}

We formally state that for all $n,m\in\N$,
    \begin{enumerate}
        \item 
        \begin{equation*}
            \sqint_{f(r)}^{g(r)}(\alpha x^n+\beta x^m)d_{s,t}x=\alpha\sqint_{f(r)}^{g(r)}x^nd_{s,t}x+\beta\sqint_{f(r)}^{g(r)}x^md_{s,t}x.
        \end{equation*}
        \item 
        \begin{equation*}
            \sqint_{f(r)}^{g(r)}\sum_{n=0}^{\infty}a_{n}x^{n}d_{s,t}x=\sum_{n=0}^{\infty}a_{n}\sqint_{f(r)}^{g(r)}x^nd_{s,t}x.
        \end{equation*}
    \end{enumerate}

\begin{theorem}[{\bf Substitution formula}]\label{theo_subst}
For all $n\in\N$,
\begin{enumerate}
    \item 
    \begin{equation}\label{eqn_sub_form1}
        \int_{a}^{b}f^{[n]}(x)\mathbf{D}_{s,t}f(x)d_{s,t}x=\sqint_{f(a)}^{f(b)}x^{n}d_{s,t}x.
    \end{equation}
    \item \begin{equation}\label{eqn_sub_form2}
    \int_{a}^bf[g(x),u]\mathbf{D}_{s,t}g(x)d_{s,t}x=\sqint_{g(a)}^{g(b)}f(x,u)d_{s,t}x.
\end{equation}
\end{enumerate}
\end{theorem}
\begin{proof}
From Definition \ref{def_sub}
\begin{align*}
    \int_{a}^{b}f^{[n]}(x)\mathbf{D}_{s,t}f(x)d_{s,t}x&=\frac{1}{\brk[c]{n+1}_{s,t}}\int_{a}^{b}\mathbf{D}_{s,t}f^{[n+1]}(x)d_{s,t}x\\
    &=\frac{f^{[n+1]}(b)-f^{[n+1]}(a)}{\brk[c]{n+1}_{s,t}}=\sqint_{f(a)}^{f(b)}x^{n}d_{s,t}x.
\end{align*}
By using Eq.(\ref{eqn_sub_form1}), we have formally
    \begin{align*}
        \int_{a}^{b}f[g(x),u]\mathbf{D}_{s,t}g(x)d_{s,t}x&=\int_{a}^{b}\left(\sum_{m=0}^{\infty}u^{\binom{m}{2}}f_{m}\frac{g^{[m]}(x)}{\brk[c]{m}_{s,t}!}\mathbf{D}_{s,t}g(x)\right)d_{s,t}x\\
        &=\sum_{m=0}^{\infty}\frac{u^{\binom{m}{2}}}{\brk[c]{m}_{s,t}!}f_{m}\int_{a}^{b}g^{[m]}(x)\mathbf{D}_{s,t}g(x)d_{s,t}x\\
        &=\sum_{m=0}^{\infty}\frac{u^{\binom{m}{2}}}{\brk[c]{m}_{s,t}!}f_{m}\sqint_{g(a)}^{g(b)}x^{m}d_{s,t}x\\
        &=\sqint_{g(a)}^{g(b)}\sum_{m=0}^{\infty}u^{\binom{m}{2}}f_{m}\frac{x^{m}}{\brk[c]{m}_{s,t}!}d_{s,t}x\\
        &=\sqint_{g(a)}^{g(b)}f(x,u)d_{s,t}x.
    \end{align*}
\end{proof}
From Eq.(\ref{eqn_comp}) we have the following composition
\begin{align*}
    \exp_{s,t}[f(x),u]&=\sum_{n=0}^{\infty}u^{\binom{n}{2}}\frac{f^{[n]}(x)}{\brk[c]{n}_{s,t}!}.
\end{align*}
Also, denote
\begin{align*}
    \Exp_{s,t}[f(x)]&=\exp_{s,t}[f(x),\varphi_{s,t}],\\
    \Exp^{\prime}_{s,t}[f(x)]&=\exp_{s,t}[f(x),\varphi_{s,t}^{\prime}]
\end{align*}
Then, $\Exp_{s,t}[x]=\Exp_{s,t}(x)$ and $\Exp_{s,t}^{\prime}[x]=\Exp^{\prime}_{s,t}(x)$. Also
\begin{align}
    \mathbf{D}_{s,t}\Exp_{s,t}[f(x)]&=\Exp_{s,t}[\varphi_{s,t}f(x)]\mathbf{D}_{s,t}f(x),\\
    \mathbf{D}_{s,t}\Exp_{s,t}^\prime[f(x)]&=\Exp_{s,t}^\prime[\varphi_{s,t}^\prime f(x)]\mathbf{D}_{s,t}f(x)
\end{align}
and in general
\begin{align*}
    \mathbf{D}_{s,t}\exp_{s,t}[f(x),u]&=\exp_{s,t}[uf(x),u]\mathbf{D}_{s,t}f(x),\\
    \mathbf{D}_{s,t}\EE_{s,t}[a,b;f(x),u]&=\Big(a\EE_{s,t}[a,b;g(x),u]
    +b\EE_{s,t}[a,b;ug(x),u]\Big)\mathbf{D}_{s,t}f(x).
\end{align*}
\begin{example}
On the one hand,
\begin{align*}
    \brk[c]{n}_{s,t}\int_{0}^{x}r^{n-1}\exp_{s,t}[ur^n,u]d_{s,t}r&=\int_{0}^{x}\mathbf{D}_{s,t}\exp_{s,t}[r^n,u]d_{s,t}r\\
    &=\exp_{s,t}[x^n,u]-1
\end{align*}
and on the other hand,
\begin{align*}
    \brk[c]{n}_{s,t}\int_{0}^{x}r^{n-1}\exp_{s,t}[ur^n,u]d_{s,t}r&=\sqint_{0}^{x^n}\exp_{s,t}(ur,u)d_{s,t}\\
    &=\exp_{s,t}[x^n,u]-1.
\end{align*}
\end{example}

\section{First order functional-difference $(s,t)$-equation}

\subsection{An equation general}

\begin{theorem}
Set $a\neq-b$ and $q=\varphi_{s,t}^{\prime}/\varphi_{s,t}$. Assume $0<\vert q\vert<1$ and suppose $\alpha$ and $\beta$ are continuous functions on an interval $[\eta,x]$, $\eta>0$. Set $A(x)=\int_{\eta}^{x}\alpha(r)d_{q}r$. Then the function $y(x)$ given by
\begin{equation}\label{eqn_sol1_gen}
    y(x)=\frac{1}{\EE_{s,t}[a,b;A(x),u]}\left(\int_{\eta}^{x}\beta(r)\EE_{s,t}[a,b;A(\varphi_{s,t}r),u]d_{s,t}r+G(\log_{q}(x))\right)
\end{equation}
is a solution to the equation
\begin{equation}\label{eqn_general}
\mathbf{D}_{s,t}y+\alpha(x)\frac{(\mathbf{D}_{s,t}E_{s,t})(a,b;x,u)\square A(x)}{\EE_{s,t}[a,b;A(\varphi_{s,t}x),u]}y(\varphi_{s,t}^{\prime}x)=\beta(x),
\end{equation}
where
\begin{equation*}
    (\mathbf{D}_{s,t}E_{s,t})(a,b;x,u)\square A(x)=a\EE_{s,t}[a,b;A(x),u]+b\EE_{s,t}[a,b;uA(x),u],
\end{equation*}
with initial value problem $y(\eta)=G(\log_{q}(\eta))$ for some $G(\log_{q}(x))\in\Per_{s,t}$. If $\eta\geq0$, then the function $y(x)$ given by
\begin{equation}\label{eqn_sol2_gen}
    y(x)=\frac{1}{\EE_{s,t}[a,b;A(x),u]}\left(\int_{\eta}^{x}\beta(r)\EE_{s,t}[a,b;A(\varphi_{s,t}r),u]d_{s,t}r\right)
\end{equation}
is a solution to the Eq.(\ref{eqn_general}) with initial value problem $y(\eta)=\xi$, for $\eta\in\R$.
\end{theorem}
\begin{proof}
By multiplying Eq.(\ref{eqn_general}) by $\EE_{s,t}[a,b;A(\varphi_{s,t}x),u]$, we obtain
\begin{align*}
\EE_{s,t}[a,b;A(\varphi_{s,t}x),u]\mathbf{D}_{s,t}y
&+\alpha(x)\left(a\EE_{s,t}[a,b;A(x),u]+b\EE_{s,t}[a,b;uA(x),u]\right)y(\varphi_{s,t}^{\prime}x)\\
&=\beta(x)\EE_{s,t}[a,b;A(\varphi_{s,t}x),u].
\end{align*}
As 
\begin{align*}
    &\mathbf{D}_{s,t}\left(\EE_{s,t}[a,b;A(x),u]y\right)(x)=\EE_{s,t}[a,b;A(\varphi_{s,t}x),u]\mathbf{D}_{s,t}y(x)\\
    &\hspace{2cm}+\alpha(x)\left(a\EE_{s,t}[a,b;A(x),u]+b\EE_{s,t}[a,b;uA(x),u]\right)y(\varphi_{s,t}^{\prime}x),
\end{align*}
then
\begin{align*}
    \mathbf{D}_{s,t}\left(\EE_{s,t}[a,b;A(x),u]y\right)=\beta(x)\EE_{s,t}[a,b;A(\varphi_{s,t}x),u].
\end{align*}
From Theorem \ref{theo_funda}
\begin{align*}
    \int_{\eta}^{x}\mathbf{D}_{s,t}\left(\EE_{s,t}[a,b;A(r),u]y\right)d_{s,t}r&=\EE_{s,t}[a,b;A(x),u]y(x)-G(\log_{q}(\eta))\\
    &=\int_{\eta}^{x}\beta(r)\EE_{s,t}[a,b;A(\varphi_{s,t}r),u]d_{s,t}r
\end{align*}
and from this, it follows that
\begin{align*}
    y(x)&=\frac{G(\log_{q}(\eta))}{\EE_{s,t}[a,b;A(x),u]}+\frac{1}{\EE_{s,t}[a,b;A(x),u]}\int_{\eta}^{x}\beta(r)\EE_{s,t}[a,b;A(\varphi_{s,t}r),u]d_{s,t}r
\end{align*}
is solution of Eq.(\ref{eqn_general}) if $\eta>0$. By changing $G(\log_{q}(x))$ by $G(\log_{q}(\eta))$ above, then we obtain a more general solution. If $\eta\geq0$, then the solution of Eq.(\ref{eqn_sol2_gen}) is
\begin{align*}
    y(x)&=\frac{\xi}{\EE_{s,t}[a,b;A(x),u]}+\frac{1}{\EE_{s,t}[a,b;A(x),u]}\int_{\eta}^{x}\beta(r)\EE_{s,t}[a,b;A(\varphi_{s,t}r),u]d_{s,t}r
\end{align*}
and the proof is reached.
\end{proof}
By interchanging $\varphi_{s,t}$ and $\varphi_{s,t}^{\prime}$ in the above theorem, we obtain the solution to the equation
\begin{equation}
\mathbf{D}_{s,t}y+\alpha(x)\frac{(\mathbf{D}_{s,t}E_{s,t})(a,b;x,u)\square A(x)}{\EE_{s,t}[a,b;A(\varphi_{s,t}^{\prime}x),u]}y(\varphi_{s,t}x)=\beta(x).
\end{equation}

\subsubsection{Equation $\mathbf{D}_{s,t}y(x)+\alpha(x)\frac{\exp_{s,t}[uA(x),u]}{\exp_{s,t}[A(\varphi_{s,t}x),u]}y(\varphi_{s,t}^{\prime}x)=\beta(x)$}

\begin{corollary}
Set $q=\varphi_{s,t}^{\prime}/\varphi_{s,t}$. Assume $0<\vert q\vert<1$ and suppose $\alpha$ and $\beta$ are continuous functions on an interval $[\eta,x]$, $\eta>0$. Set $A(x)=\int_{\eta}^{x}\alpha(r)d_{q}r$. Then the function $y(x)$ given by
\begin{equation}\label{eqn_sol_c1}
    y(x)=\frac{1}{\exp_{s,t}[A(x),u]}\left(\int_{\eta}^{x}\beta(r)\exp_{s,t}[A(\varphi_{s,t}r),u]d_{s,t}r+G(\log_{q}(x))\right)
\end{equation}
is a solution to the initial value problem
\begin{equation}
\mathbf{D}_{s,t}y+\alpha(x)\frac{\exp_{s,t}[uA(x),u]}{\exp_{s,t}[A(\varphi_{s,t}x),u]}y(\varphi_{s,t}^{\prime}x)=\beta(x),\ \ y(\eta)=G(\log_{q}(\eta))
\end{equation}
for some $G(\log_{q}(x))\in\Per_{s,t}$. If $\eta\geq0$, then the function $y(x)$ given by
\begin{equation}
    y(x)=\frac{1}{\exp_{s,t}[A(x),u]}\left(\int_{\eta}^{x}\beta(r)\exp_{s,t}[A(\varphi_{s,t}r),u]d_{s,t}r+\xi\right),
\end{equation}
is a solution to the initial value problem
\begin{equation}
\mathbf{D}_{s,t}y+\alpha(x)\frac{\exp_{s,t}[uA(x),u]}{\exp_{s,t}[A(\varphi_{s,t}x),u]}y(\varphi_{s,t}^{\prime}x)=\beta(x),\ \ y(\eta)=\xi
\end{equation}
for $\eta\in\R$.
\end{corollary}

\begin{example}
Let us now take into account the equation
\begin{equation*}
    \mathbf{D}_{s,t}y(x)+x^{n-1}\frac{\exp_{s,t}[ux^{n},u]}{\exp_{s,t}[\varphi_{s,t}^nx^n,u]}y(\varphi_{s,t}^\prime x)=x^{n-1},\ y(0)=\xi
\end{equation*}
where $n\neq0$. From Eq.(\ref{eqn_sol_c1}),
\begin{align*}
    y(x)&=\frac{1}{\exp_{s,t}[x^n,u]}\left(\int_{0}^x r^{n-1}\exp_{s,t}[ur^{n},u]d_{s,t}r+\xi\right)\\
    &=\frac{1}{\exp_{s,t}[x^n,u]}\left(\frac{\exp_{s,t}[x^n,u]-1}{\brk[c]{n}_{s,t}}+\xi\right)\\
    &=\frac{\exp_{s,t}[x^n,u]-1}{\brk[c]{n}_{s,t}\exp_{s,t}[x^n,u]}+\frac{\xi}{\exp_{s,t}[x^n,u]}.
\end{align*}
\end{example}

\subsubsection{Solution of the equation $\mathbf{D}_{s,t}y+\alpha(x)y(\varphi_{s,t}^{\prime}x)=\beta(x)$}

\begin{corollary}
Set $q=\varphi_{s,t}^{\prime}/\varphi_{s,t}$. Assume $0<\vert q\vert<1$ and suppose $\alpha$ and $\beta$ are continuous functions on an interval $[\eta,x]$, $\eta>0$. Set $A(x)=\int_{\eta}^{x}\alpha(r)d_{q}r$. Then the function $y(x)$ given by
\begin{equation}\label{eqn_sol1_qf}
    y(x)=\frac{1}{\Exp_{s,t}[A(x)]}\left(\int_{\eta}^{x}\beta(r)\Exp_{s,t}[A(\varphi_{s,t}r)]d_{s,t}r+G(\log_{q}(x))\right)
\end{equation}
is a solution to the initial value problem
\begin{equation}\label{eqn_lin}
\mathbf{D}_{s,t}y+\alpha(x)y(\varphi_{s,t}^{\prime}x)=\beta(x),\ \ y(\eta)=G(\log_{q}(\eta))
\end{equation}
for some $G(\log_{q}(x))\in\Per_{s,t}$. If $\eta\geq0$, then the function $y(x)$ given by
\begin{equation}
    y(x)=\frac{1}{\Exp_{s,t}[A(x)]}\left(\int_{\eta}^{x}\beta(r)\Exp_{s,t}[A(\varphi_{s,t}r)]d_{s,t}r+\xi\right),
\end{equation}
is a solution to the initial value problem
\begin{equation}\label{eqn_lin2}
\mathbf{D}_{s,t}y+\alpha(x)y(\varphi_{s,t}^{\prime}x)=\beta(x),\ \ y(\eta)=\xi
\end{equation}
for $\eta\in\R$.
\end{corollary}

\begin{example}\label{exam2}
Let us consider the equation
\begin{equation*}
    \mathbf{D}_{s,t}y(x)-y(\varphi_{s,t}^{\prime}x)=x^{m}
\end{equation*}
where $m\geq0$. From Eq.(\ref{eqn_sol1_qf})
\begin{align*}
    y(x)&=\frac{1}{\Exp_{s,t}[-x]}\left(\int r^{m}\Exp_{s,t}[-\varphi_{s,t}r]d_{s,t}r+G(\log_{q}(x))\right)\\
    &=\frac{1}{\Exp_{s,t}(-x)}\left(\int r^{m}\Exp_{s,t}(-\varphi_{s,t}r)d_{s,t}r+G(\log_{q}(x))\right)
\end{align*}
and from Eq.(\ref{theo_partes2})
\begin{align*}
    \int x^{m}\Exp_{s,t}(-\varphi_{s,t}x)d_{s,t}x&=-\varphi_{s,t}^{\prime(-m)}\int(\varphi_{s,t}^{\prime}x)^{m}\mathbf{D}_{s,t}\Exp_{s,t}(-x)d_{s,t}x\\
    &=-\varphi_{s,t}^{\prime(-m)}x^{m}\Exp_{s,t}(-x)\\
    &\hspace{2cm}+\varphi_{s,t}^{\prime(-m)}\brk[c]{m}_{s,t}\int x^{m-1}\Exp_{s,t}(-\varphi_{s,t}x)d_{s,t}x.
\end{align*}
Then 
\begin{equation*}
    \int x^{m}\Exp_{s,t}(-\varphi_{s,t}x)d_{s,t}x=-\left(\sum_{k=0}^{m}\varphi_{s,t}^{\prime(-m(k+1)+\binom{k+1}{2})}\frac{\brk[c]{m}_{s,t}!}{\brk[c]{m-k}_{s,t}!}x^{m-k}\right)\Exp_{s,t}(-x)
\end{equation*}
and
\begin{equation}
    y(x)=G(\log_{q}(x))\Exp_{s,t}^{\prime}(x)-\sum_{k=0}^{m}\varphi_{s,t}^{\prime(-m(k+1)+\binom{k+1}{2})}\frac{\brk[c]{m}_{s,t}!}{\brk[c]{m-k}_{s,t}!}x^{m-k}.
\end{equation}
\end{example}

\begin{example}
Let us consider the equation
\begin{equation*}
    \mathbf{D}_{s,t}y(x)+y(\varphi_{s,t}x)=\alpha\Exp_{s,t}^{\prime}(\beta x),
\end{equation*}
where $\beta\neq-\varphi_{s,t}^\prime$. From Eq.(\ref{eqn_sol1_qf})
\begin{align*}
    y(x)&=\frac{1}{\Exp_{s,t}(x)}\left(\alpha\int\Exp_{s,t}^{\prime}(\beta x)\Exp_{s,t}(\varphi_{s,t}x)d_{s,t}x+G(\log_{q}(x))\right)\\
    &=\frac{1}{\Exp_{s,t}(x)}\left(\alpha\int\exp_{s,t}((\varphi_{s,t}\oplus\beta)_{\varphi,\varphi^\prime}x)d_{s,t}x+G(\log_{q}(x))\right)\\
    &=\frac{\alpha\varphi_{s,t}^\prime}{\varphi_{s,t}^{\prime}+\beta}\Exp_{s,t}^{\prime}(\beta x/\varphi_{s,t}^\prime)+\frac{G(\log_{q}(x))}{\Exp_{s,t}(x)}.
\end{align*}
\end{example}

\begin{example}
Let us consider the equation
\begin{equation*}
    \mathbf{D}_{s,t}y(x)+y(\varphi_{s,t}x)=\alpha x^{m}\Exp_{s,t}^{\prime}(-\varphi_{s,t}x)
\end{equation*}
From Eq.(\ref{eqn_sol1_qf}),
\begin{align*}
    y(x)&=\frac{1}{\Exp_{s,t}(x)}\left(\alpha\int x^{m}\Exp_{s,t}^{\prime}(-\varphi_{s,t}x)\Exp_{s,t}(\varphi_{s,t}x)d_{s,t}x+G(\log_{q}(x))\right)\\
    &=\frac{1}{\Exp_{s,t}(x)}\left(\alpha\int x^{m}d_{s,t}x+G(\log_{q}(x))\right)\\
    &=\frac{\alpha}{\brk[c]{m+1}_{s,t}}\frac{x^{m+1}}{\Exp_{s,t}(x)}+\frac{G(\log_{q}(x))}{\Exp_{s,t}(x)}
\end{align*}
\end{example}

\subsubsection{Equation involving partial theta function}

\begin{corollary}
Set $q=\varphi_{s,t}^{\prime}/\varphi_{s,t}$. Assume $0<\vert q\vert<1$ and suppose $\alpha$ and $\beta$ are continuous functions on an interval $[\eta,x]$, $\eta>0$. Set $A(x)=\int_{\eta}^{x}\alpha(r)d_{q}r$. Then the function $y(x)$ given by
\begin{equation}
    y(x)=\frac{1}{\Theta_{0}[(1-q)A(x),\varphi_{s,t}^{-1}]}\left(\int_{\eta}^{x}\beta(r)\Theta_{0}[(1-q)A(\varphi_{s,t}r),\varphi_{s,t}^{-1}]d_{s,t}r+G(\log_{q}(x))\right)
\end{equation}
is a solution to the initial value problem
\begin{equation}
\mathbf{D}_{s,t}y(x)
+\alpha(\varphi_{s,t}x)\frac{(D_{q}\Theta_{0}((1-q)x,\varphi_{s,t}^{-1}))\square A(\varphi_{s,t}x)}{\Theta_{0}[(1-q)A(\varphi_{s,t}x),\varphi_{s,t}^{-1}]}y(\varphi_{s,t}^{\prime}x)
=\beta(x),
\end{equation}
$y(\eta)=G(\log_{q}(\eta))$, for some $G(\log_{q}(x))\in\Per_{s,t}$. If $\eta\geq0$, then the function $y(x)$ given by
\begin{equation}
    y(x)=\frac{1}{\Theta_{0}[(1-q)A(x),\varphi_{s,t}^{-1}]}\left(\int_{\eta}^{x}\beta(r)\Theta_{0}[(1-q)A(\varphi_{s,t}r),\varphi_{s,t}^{-1}]d_{s,t}r+\xi\right)
\end{equation}
is a solution to the initial value problem
\begin{equation}
\mathbf{D}_{s,t}y+\alpha(\varphi_{s,t}x)\frac{(D_{q}\Theta_{0}((1-q)x,\varphi_{s,t}^{-1}))\square A(\varphi_{s,t}x)}{\Theta_{0}[(1-q)A(\varphi_{s,t}x),\varphi_{s,t}^{-1}]}y(\varphi_{s,t}^{\prime}x)=\beta(x),\ \ y(\eta)=\xi
\end{equation}
for $\eta\in\R$.
\end{corollary}

\begin{example}
The initial value problem
\begin{equation*}
    \mathbf{D}_{s,t}y(x)+\frac{\mathbf{D}_{s,t}\Theta_{0}((1-q)x,\varphi_{s,t}^{-1})}{\Theta_{0}((1-q)\varphi_{s,t}x,\varphi_{s,t}^{-1})}y(\varphi_{s,t}^{\prime}x)=1,\hspace{1cm}y(0)=\xi
\end{equation*}
has solution
\begin{align*}
    y(x)&=\frac{1}{\Theta_{0}((1-q)x,\varphi_{s,t}^{-1})}\int_{0}^{x}\Theta_{0}[(1-q)\varphi_{s,t}r,\varphi_{s,t}^{-1}]d_{s,t}r+\frac{\xi}{\Theta_{0}((1-q)x,\varphi_{s,t}^{-1})}\\
    &=\frac{1}{\varphi_{s,t}\Theta_{0}((1-q)x,\varphi_{s,t}^{-1})}\sum_{k=0}^{\infty}(\Theta_{0}((1-q)q^{k}\varphi_{s,t}x,\varphi_{s,t}^{-1})-1)+\frac{\xi}{\Theta_{0}((1-q)x,\varphi_{s,t}^{-1})}.
\end{align*}
If $x=\frac{\varphi_{s,t}^{-1}}{1-q}$, then
\begin{equation*}
    y\left(\frac{\varphi_{s,t}^{-1}}{1-q}\right)=\frac{(\varphi_{s,t}^{-1};\varphi_{s,t}^{-2})_{\infty}}{(\varphi_{s,t}^{-2};\varphi_{s,t}^{-2})_{\infty}}\sum_{k=0}^{\infty}(\Theta_{0}(q^{k},\varphi_{s,t}^{-1})-1)+\frac{(\varphi_{s,t}^{-1};\varphi_{s,t}^{-2})_{\infty}}{(\varphi_{s,t}^{-2};\varphi_{s,t}^{-2})_{\infty}}\xi.
\end{equation*}
\end{example}

\subsection{Equation $\mathbf{D}_{s,t}y=\alpha(ay(x)+by(ux))+\beta(x)$}

\begin{theorem}
Consider the following equation
\begin{equation}\label{eqn_ab_b(x)}
    \mathbf{D}_{s,t}y(x)=\alpha(ay(x)+by(u x))+\beta(x),\ y(0)=a_{0},
\end{equation}
where $\alpha$  is a constant, and $\beta(x)$ is an analytic function with expansion
\begin{equation*}
    \beta(x)=\sum_{n=0}^{\infty}b_{n}x^n.
\end{equation*}
Then its solution is given by
\begin{equation}\label{eqn_sol_ab_b(x)}
    y(x)=a_{0}\EE_{s,t}(a,b;\alpha x,u)+\sum_{n=1}^{\infty}\left(\sum_{k=0}^{n-1}\frac{\brk[c]{k}_{s,t}!\alpha^{n-k-1}}{(a\oplus b)_{1,u}^{k+1}}b_{k}\right)(a\oplus b)_{1,u}^{n}\frac{x^n}{\brk[c]{n}_{s,t}!}.
\end{equation}
\end{theorem}
\begin{proof}
Assuming that the power series expansion of $y(x)$ is
\begin{equation*}
    y(x)=\sum_{n=0}^{\infty}a_{n}x^{n},
\end{equation*}
and substituting it into the above equation and setting all coefficients of each term $x^n$ to be zeros, we obtain the formulas of $a_{n}$ as follows
\begin{equation*}
    a_{n+1}=\frac{a+bu^{n}}{\brk[c]{n+1}_{s,t}}\alpha a_{n}+\frac{b_{n}}{\brk[c]{n+1}_{s,t}}.
\end{equation*}
For $n\geq1$, we have
\begin{equation*}
    a_{n}=\frac{(a\oplus b)_{1,u}^{n}}{\brk[c]{n}_{s,t}!}\alpha^{n}a_{0}+\frac{(a\oplus b)_{1,u}^{n}}{\brk[c]{n}_{s,t}!}\sum_{k=0}^{n-1}\frac{\brk[c]{k}_{s,t}!\alpha^{n-1-k}}{(a\oplus b)_{1,u}^{k+1}}b_{k}.
\end{equation*}
So the solution can be represented by
\begin{align*}
    y(x)&=a_{0}\sum_{n=0}^{\infty}\frac{(a\oplus b)_{1,u}^{n}}{\brk[c]{n}_{s,t}!}\alpha^{n}x^n+\sum_{n=1}^{\infty}\frac{(a\oplus b)_{1,u}^{n}}{\brk[c]{n}_{s,t}!}\left(\sum_{k=0}^{n-1}\frac{\brk[c]{k}_{s,t}!\alpha^{n-k-1}}{(a\oplus b)_{1,u}^{k+1}}b_{k}\right)x^{n}\\
    &=a_{0}\EE_{s,t}(a,b;\alpha x,u)+\sum_{n=1}^{\infty}\frac{(a\oplus b)_{1,u}^{n}}{\brk[c]{n}_{s,t}!}\left(\sum_{k=0}^{n-1}\frac{\brk[c]{k}_{s,t}!\alpha^{n-k-1}}{(a\oplus b)_{1,u}^{k+1}}b_{k}\right)x^{n}.
\end{align*}
The proof is completed.
\end{proof}

\begin{corollary}\label{cor_ulin_eqn}
Consider the following equation
\begin{equation}\label{eqn_ab(x)}
    \mathbf{D}_{s,t}y(x)=\alpha y(ux)+\beta(x),\ y(0)=a_{0},
\end{equation}
where $\alpha$  is a constant, and $\beta(x)$ is an analytic function with expansion
\begin{equation*}
    \beta(x)=\sum_{n=0}^{\infty}b_{n}x^n.
\end{equation*}
Then its solution is given by
\begin{equation}\label{eqn_sol_ab(x)}
    y(x)=a_{0}\exp_{s,t}(\alpha x,u)+\sum_{n=1}^{\infty}\left(\sum_{k=0}^{n-1}\frac{\brk[c]{k}_{s,t}!\alpha^{n-1-k}}{u^{\binom{k+1}{2}}}b_{k}\right)u^{\binom{n}{2}}\frac{x^n}{\brk[c]{n}_{s,t}!}.
\end{equation}
\end{corollary}

\begin{corollary}
Set $q=\varphi_{s,t}^{\prime}/\varphi_{s,t}$. Consider the following equation
\begin{equation}
    \mathbf{D}_{s,t}y(x)=y(x)-qy(qx)+\beta(x),\ y(0)=a_{0},
\end{equation}
where $\alpha$  is a constant, and $\beta(x)$ is an analytic function with expansion
\begin{equation*}
    \beta(x)=\sum_{n=0}^{\infty}b_{n}x^n.
\end{equation*}
Then its solution is given by
\begin{equation}
    y(x)=a_{0}\Theta_{0}((1-q))x,\varphi_{s,t}^{-1})+\sum_{n=1}^{\infty}\left(\sum_{k=0}^{n-1}\frac{\varphi_{s,t}^{\binom{k}{2}}}{(1-q)^{k}(1-q^{k+1})}b_{k}\right)(1-q)^{n}\varphi_{s,t}^{-\binom{n}{2}}x^n.
\end{equation}
\end{corollary}

\begin{example}\label{exam7}
The solution of
\begin{equation*}
    \mathbf{D}_{s,t}y(x)=\alpha(ay(x)+by(ux))+\delta+\epsilon x,\ y(0)=a_{0}
\end{equation*}
is     
\begin{align*}
    y(x)&=\left(c+\frac{\delta}{\alpha(a\oplus b)_{1,u}^{1}}+\frac{\epsilon}{\alpha^{2}(a\oplus b)_{1,u}^{2}}\right)\EE_{s,t}(a,b;\alpha x,u)-\frac{\delta}{\alpha(a\oplus b)_{1,u}^{1}}-\frac{\epsilon}{\alpha^{2}(a\oplus b)_{1,u}^{2}}
\end{align*}
where $c=a_{0}$.
\end{example}

\begin{example}
The solution of
\begin{equation*}
    \mathbf{D}_{s,t}y(x)=\alpha(ay(x)+by(ux))+x^{m},\ y(0)=a_{0}
\end{equation*}
is 
\begin{equation*}
    y(x)=c\EE_{s,t}(a,b;\alpha x,u)+\frac{\brk[c]{m}_{s,t}!}{\alpha^{m+1}(a\oplus b)^{m+1}_{1,u}}\left(\EE_{s,t}(a,b;\alpha x,u)-\sum_{n=0}^{m}\frac{(a\oplus b)^{n}_{1,u}}{\brk[c]{n}_{s,t}!}(\alpha x)^n\right)
\end{equation*}
where $c=a_{0}$.
\end{example}

\begin{theorem}
The solution of 
\begin{equation*}
    \mathbf{D}_{s,t}y(x)=\varphi_{s,t}^{\prime}(ay(x)+by(u x))+\beta(a\EE_{s,t}(a,b,u;\varphi_{s,t}x)+b\EE_{s,t}(a,b,u;\varphi_{s,t}ux))
\end{equation*}
is given by 
\begin{equation*}
    y(x)=c\EE_{s,t}(a,b;\varphi_{s,t}^{\prime}x,u)+\beta x(a\EE_{s,t}(a,b;x,u)+b\EE_{s,t}(a,b;u x,u)),
\end{equation*}
where $c=a_{0}$.
\end{theorem}
\begin{proof}
The power series expansion of $\beta(a\EE_{s,t}(a,b;\varphi_{s,t}x,u)+b\EE_{s,t}(a,b;\varphi_{s,t}ux,u))$ is
\begin{equation*}
    \beta(a\EE_{s,t}(a,b;\varphi_{s,t}x,u)+b\EE_{s,t}(a,b;\varphi_{s,t}ux,u))=\beta\sum_{k=0}^{\infty}(a\oplus b)_{1,u}^{k+1}\varphi_{s,t}^{k}\frac{x^k}{\brk[c]{k}_{s,t}!}
\end{equation*}
which gives
\begin{equation*}
    b_{k}=\beta\frac{(a\oplus b)_{1,u}^{k+1}}{\brk[c]{k}_{s,t}!}\varphi_{s,t}^{k}.
\end{equation*}
Substituting it into the Eq.(\ref{eqn_sol_ab_b(x)}) yields
\begin{align*}
    y(x)&=a_{0}\EE_{s,t}(a,b;\varphi_{s,t}^{\prime}x,u)+\beta\sum_{n=1}^{\infty}\left(\sum_{k=0}^{n-1}\frac{\brk[c]{k}_{s,t}!\varphi_{s,t}^{\prime(n-1-k)}}{(a\oplus b)_{1,u}^{k+1}}\frac{(a\oplus b)_{1,u}^{k+1}}{\brk[c]{k}_{s,t}!}\varphi_{s,t}^{k}\right)(a\oplus b)_{1,u}^{n}\frac{x^n}{\brk[c]{n}_{s,t}!}\\
    &=a_{0}\EE_{s,t}(a,b;\varphi_{s,t}^{\prime}x,u)+\beta\sum_{n=1}^{\infty}\left(\sum_{k=0}^{n-1}\varphi_{s,t}^{k}\varphi_{s,t}^{\prime(n-1-k)}\right)(a\oplus b)_{1,u}^{n}\frac{x^n}{\brk[c]{n}_{s,t}!}\\
    &=a_{0}\EE_{s,t}(a,b;\varphi_{s,t}^{\prime}x,u)+\beta\sum_{n=1}^{\infty}\brk[c]{n}_{s,t}(a\oplus b)_{1,u}^{n}\frac{x^n}{\brk[c]{n}_{s,t}!}\\
    &=a_{0}\EE_{s,t}(a,b;\varphi_{s,t}^{\prime}x,u)+\beta\sum_{n=1}^{\infty}(a\oplus b)_{1,u}^{n}\frac{x^n}{\brk[c]{n-1}_{s,t}!}\\
    &=a_{0}\EE_{s,t}(a,b;\varphi_{s,t}^{\prime}x,u)+\beta x\sum_{n=0}^{\infty}(a\oplus b)_{1,u}^{n+1}\frac{x^n}{\brk[c]{n}_{s,t}!}\\
    &=a_{0}\EE_{s,t}(a,b;\varphi_{s,t}^{\prime}x,u)+\beta x(a\EE_{s,t}(a,b;x,u)+b\EE_{s,t}(a,b;ux,u)).
\end{align*}
The proof is completed.
\end{proof}

\begin{corollary}
The solution of 
\begin{equation*}
    \mathbf{D}_{s,t}y(x)=\varphi_{s,t}^{\prime}y(u x)+\beta\exp_{s,t}(\varphi_{s,t}u x,u)
\end{equation*}
is given by 
\begin{equation*}
    y(x)=c\exp_{s,t}(\varphi_{s,t}^{\prime}x,u)+\beta x\exp_{s,t}(u x,u),
\end{equation*}
where $c=a_{0}$.
\end{corollary}

\begin{corollary}
The solution of 
\begin{equation*}
    \mathbf{D}_{s,t}y(x)=\varphi_{s,t}^{\prime}(y(x)-qy(qx))+\beta(\Theta_{0}((1-q)\varphi_{s,t}x,\varphi_{s,t}^{-1})-q\Theta_{0}((1-q))\varphi_{s,t}qx,\varphi_{s,t}^{-1}))
\end{equation*}
is given by 
\begin{equation*}
    y(x)=c\Theta_{0}((1-q)\varphi_{s,t}^{\prime}x,\varphi_{s,t}^{-1})+\beta x(\Theta_{0}((1-q))x,\varphi_{s,t}^{-1})-q\Theta_{0}((1-q))qx,\varphi_{s,t}^{-1})),
\end{equation*}
where $c=a_{0}$.
\end{corollary}

\subsection{Operator method}
We use an operator method to deal with a linear proportional difference equation. Two basic operator are $(s,t)$-derivative operator $\mathbf{D}_{s,t}$ and the scale $T_{u}$ defined by
\begin{equation*}
    T_{u}y(x)=y(u x).
\end{equation*}
Two basic properties of these two operators are
\begin{align*}
    \mathbf{D}_{s,t}T_{u}&=uT_{u}\mathbf{D}_{s,t},\\
    T_{u}T_{v}&=T_{uv}.
\end{align*}

\begin{theorem}
For the first-order equation 
\begin{equation}
    \mathbf{D}_{s,t}y(x)=a\beta y(x)+ \gamma T_{u}y(x)+\delta\EE_{s,t}(a,b;\alpha x,u)
\end{equation}
the general solution is given by
\begin{equation}
    y(x)=c\EE_{s,t}(a\beta,\gamma;x,u)+\frac{u\delta}{b\alpha-u\gamma}\EE_{s,t}(a,b;\alpha x/u,u)
\end{equation}
\end{theorem}
\begin{proof}
Formally, we have 
\begin{equation*}
    y(x)=\delta(\mathbf{D}_{s,t}-a\beta1-\gamma T_{u})^{-1}\EE_{s,t}(a,b;\alpha x,u).
\end{equation*}
To compute the right side of the above equation, we use the following formula 
\begin{equation*}
    (\mathbf{D}_{s,t}-a\beta1-\gamma T_{u})\EE_{s,t}(a,b;\beta x,u)=(b\beta-\gamma)\EE_{s,t}(a,b;u\beta x,u).
\end{equation*}
So, taking $\beta=\frac{\alpha}{u}$ gives
\begin{equation*}
    (\mathbf{D}_{s,t}-a\frac{\alpha}{u}1-\gamma T_{u})^{-1}\EE_{s,t}(a,b;\alpha x,u)=\frac{1}{\frac{b\alpha}{u}-\gamma}\EE_{s,t}(a,b;\alpha x/u,u).
\end{equation*}
Therefore, a special solution is given by 
\begin{equation*}
    y^{\ast}(x)=\frac{u\delta}{b\alpha-u\gamma}\EE_{s,t}(a,b;\alpha x/u,u)
\end{equation*}
and the general solution is 
\begin{equation*}
    y(x)=c\EE_{s,t}(a\beta,\gamma; x,u)+\frac{u\delta}{b\alpha-u\gamma}\EE_{s,t}(a,b;\alpha x/u,u).
\end{equation*}
\end{proof}

\subsection{Deformed $(s,t)$-Bernoulli Equation}
In this section we will focus on solving $(s,t)$-equations in the form
\begin{multline}\label{eqn_berno1}
    D_{\varphi^{n-1},\varphi^{\prime(n-1)}}y(x)+\alpha(x)\frac{(D_{q}E_{s,t})(a,b;x,u)\square A(x)}{\EE_{s,t}[a,b;A(\varphi_{s,t}x),u]}y(\varphi_{s,t}^{n-1}x)\\
    =\beta(x)\prod_{j=0}^{\infty}\frac{y(q^{j}\varphi_{s,t}^{n-1}x)}{y(q^{n+j}\varphi_{s,t}^{n-1}x)}
\end{multline}
and
\begin{multline}\label{eqn_berno2}
    D_{\varphi^{n-1},\varphi^{\prime(n-1)}}y(x)+\alpha(x)\frac{(D_{q}E_{s,t})(a,b;x,u)\square A(x)}{\EE_{s,t}[a,b;A(\varphi_{s,t}^{\prime}x),u]}y(\varphi_{s,t}^{\prime(n-1)}x)\\=\beta(x)\prod_{j=0}^{\infty}\frac{y(q^{j}\varphi_{s,t}^{n-1}x)}{y(q^{n+j}\varphi_{s,t}^{n-1}x)}
\end{multline}
where $\alpha(x)$ and $\beta(x)$ are continuous functions, $n$ is any real number, and where
\begin{equation*}
    D_{\varphi^{n-1},\varphi^{\prime(n-1)}}f(x)=\frac{f(\varphi_{s,t}^{n-1}x)-f(\varphi_{s,t}^{\prime(n-1)}x)}{(\varphi_{s,t}^{n-1}-\varphi_{s,t}^{\prime(n-1)})x}
\end{equation*}
is a $(p,q)$-derivative. Note that $n\neq0$ and $n\neq1$. Also, note that when $s=2$ and $t=-1$, those $(s,t)$-difference equations Eqs.(\ref{eqn_berno1},\ref{eqn_berno2}) lead to the usual Bernoulli equations.

To solve $(s,t)$-difference equation Eq.(\ref{eqn_berno1}), it is first divided by $\prod_{j=0}^{\infty}\frac{y(q^{j}\varphi_{s,t}^{n-1}x)}{y(q^{n+j}\varphi_{s,t}^{n-1}x)}$,
\begin{equation*}
    D_{\varphi^{n-1},\varphi^{\prime(n-1)}}y(x)\prod_{j=0}^{\infty}\frac{y(q^{n+j}\varphi_{s,t}^{n-1}x)}{y(q^{j}\varphi_{s,t}^{n-1}x)}+\alpha(x)\prod_{j=0}^{\infty}\frac{y(q^{n+j}\varphi_{s,t}^{n-1}x)}{y(q^{j+1}\varphi_{s,t}^{n-1}x)}=\beta(x).
\end{equation*}
Substitution of $z(x)=\prod_{j=0}^{\infty}\frac{y(q^{n-1+j}\varphi_{s,t}^{n-2}x)}{y(q^{j}\varphi_{s,t}^{n-2}x)}$ is used to convert the above equation into a $(s,t)$-difference equation in terms of $z(x)$. The $(s,t)$-derivative of $z(x)$ is as follows
\begin{align*}
    \mathbf{D}_{s,t}z(x)&=\mathbf{D}_{s,t}\left(\prod_{j=0}^{\infty}\frac{y(q^{n-1+j}\varphi_{s,t}^{n-2}x)}{y(q^{j}\varphi_{s,t}^{n-2}x)}\right)\\
    &=\frac{1}{(\varphi_{s,t}-\varphi_{s,t}^{\prime})x}\left(\prod_{j=0}^{\infty}\frac{y(q^{n-1+j}\varphi_{s,t}^{n-1}x)}{y(q^{j}\varphi_{s,t}^{n-1}x)}-\prod_{j=0}^{\infty}\frac{y(q^{n-1+j}\varphi_{s,t}^{n-2}\varphi_{s,t}^{\prime}x)}{y(q^{j}\varphi_{s,t}^{n-2}\varphi_{s,t}^{\prime}x)}\right)\\
    &=-\frac{\varphi_{s,t}^{n-1}-\varphi_{s,t}^{\prime(n-1)}}{(\varphi_{s,t}-\varphi_{s,t}^{\prime})}\frac{y(\varphi_{s,t}^{n-1}x)-y(\varphi_{s,t}^{\prime(n-1)}x)}{(\varphi_{s,t}^{n-1}-\varphi_{s,t}^{\prime(n-1)})x}\prod_{j=0}^{\infty}\frac{y(q^{n+j}\varphi_{s,t}^{n-1}x)}{y(q^{j}\varphi_{s,t}^{n-1}x)}\\
    &=-\brk[c]{n-1}_{s,t}D_{\varphi^{n-1},\varphi^{\prime(n-1)}}y(x)\prod_{j=0}^{\infty}\frac{y(q^{n+j}\varphi_{s,t}^{n-1}x)}{y(q^{j}\varphi_{s,t}^{n-1}x)}.
\end{align*}
Then the $(s,t)$-difference equation equivalent is
\begin{equation}\label{eqn_berno1t}
    -\frac{1}{\brk[c]{n-1}_{s,t}}\mathbf{D}_{s,t}z(x)+\alpha(x)\frac{(D_{q}E_{s,t})(a,b;x,u)\square A(x)}{\EE_{s,t}[a,b;A(\varphi_{s,t}x),u]}z(\varphi_{s,t}^{\prime}x)=\beta(x).
\end{equation}
The solution $z(x)$ is used to find the required solution $y(x)$. Let us rewrite it as
\begin{equation*}
    y(x)=\begin{cases}
        y(q^{n-1}x)\frac{z(qx/\varphi_{s,t}^{n-2})}{z(x/\varphi_{s,t}^{n-2})},&\text{ if }n\neq2;\\
        \frac{1}{z(x)},&\text{ if }n=2.
    \end{cases}
\end{equation*}
The previous expression can be rewrite for $y(x)$ as follows:
\begin{enumerate}
    \item For $0<\vert q\vert<1$, $n>1$ and $n\neq2$:
    \begin{equation*}
        y(x)=y(0)\prod_{i=0}^{\infty}\frac{z(q^{i(n-1)+1}x/\varphi_{s,t}^{n-2})}{z(q^{i(n-1)}x/\varphi_{s,t}^{n-2})}.
    \end{equation*}
    \item For $\vert q\vert>1$ and $n<1$:
    \begin{equation*}
        y(x)=y(0)\prod_{i=0}^{\infty}\frac{z(q^{i(n-1)+1}x/\varphi_{s,t}^{n-2})}{z(q^{i(n-1)}x/\varphi_{s,t}^{n-2})}.
    \end{equation*}
    \item For $n=2$,
    \begin{equation*}
        y(x)=\frac{1}{z(x)}.
    \end{equation*}
    \item Otherwise
    \begin{equation*}
        y(x)=y(\infty)\prod_{i=0}^{\infty}\frac{z(q^{i(n-1)+1}x/\varphi_{s,t}^{n-2})}{z(q^{i(n-1)}x/\varphi_{s,t}^{n-2})}.
    \end{equation*}
\end{enumerate}
In a similar way, the $(s,t)$-difference equation in Eq.(\ref{eqn_berno2}) is converted to the following $(s,t)$-difference equation in term of $z(x)$
\begin{equation}\label{eqn_berno2t}
    -\frac{1}{\brk[c]{n-1}_{s,t}}\mathbf{D}_{s,t}z(x)+\alpha( x)\frac{(D_{q}E_{s,t})(a,b;x,u)\square A(x)}{\EE_{s,t}[a,b;A(\varphi_{s,t}^{\prime}x),u]}z(\varphi_{s,t}x)=\beta(x)
\end{equation}
and the required solution is found to be as the previous one. Finally, the Eqs.(\ref{eqn_berno1t},\ref{eqn_berno2t}) are solved for $z(x)$ as previously carried out.

\begin{example}
The equation
\begin{equation*}
    D_{\varphi,\varphi^\prime}y(x)+y(\varphi_{s,t}x)=x^{2}y(\varphi_{s,t}x)y(\varphi_{s,t}^{\prime}x)
\end{equation*}
is a $(s,t)$-Bernoulli difference equation for $n=2$. Let $z(x)=\frac{1}{y(x)}$. Now the $(s,t)$-difference equation is transformed on $z(x)$, to
\begin{equation*}
    \mathbf{D}_{s,t}z(x)-z(\varphi_{s,t}^{\prime}x)=-x^{m}.
\end{equation*}
From Example \ref{exam2}, the solution is
\begin{equation*}
    z(x)=G(\log_{q}(x))\Exp_{s,t}^{\prime}(x)+\sum_{k=0}^{m}\varphi_{s,t}^{\prime(-m(k+1)+\binom{k+1}{2})}\frac{\brk[c]{m}_{s,t}!}{\brk[c]{m-k}_{s,t}!}x^{m-k}.
\end{equation*}
\end{example}

Finally, we define the following $u$-deformed $(s,t)$-Bernoulli difference equations
\begin{equation}\label{eqn_uberno1}
    D_{\varphi^{n-1},\varphi^{\prime(n-1)}}y(x)+\alpha y(\varphi_{s,t}^{n-1}ux)=\beta(x)\prod_{j=0}^{\infty}\frac{y(q^{j}\varphi_{s,t}^{n-1}ux)}{y(q^{n+j}\varphi_{s,t}^{n-1}ux)}
\end{equation}
and
\begin{equation}\label{eqn_uberno2}
    D_{\varphi^{n-1},\varphi^{\prime(n-1)}}y(x)+\alpha y(\varphi_{s,t}^{\prime(n-1)}ux)=\beta(x)\prod_{j=0}^{\infty}\frac{y(q^{j}\varphi_{s,t}^{n-1}ux)}{y(q^{n+j}\varphi_{s,t}^{n-1}ux)}
\end{equation}
with initial value problem $y(0)=a_{0}$. With the substitution 
\begin{equation*}
    z(x)=\prod_{j=0}^{\infty}\frac{y(q^{n-1+j}\varphi_{s,t}^{n-2}ux)}{y(q^{j}\varphi_{s,t}^{n-2}ux)}
\end{equation*}
the Eqs.(\ref{eqn_uberno1},\ref{eqn_uberno2}) are transformed into the equations
\begin{equation*}
    -\frac{1}{\brk[c]{n-1}_{s,t}}\mathbf{D}_{s,t}z(x)+\alpha z(\varphi_{s,t}^{\prime}ux)=\beta(x)
\end{equation*}
and
\begin{equation*}
    -\frac{1}{\brk[c]{n-1}_{s,t}}\mathbf{D}_{s,t}z(x)+\alpha z(\varphi_{s,t}ux)=\beta(x)
\end{equation*}
which can be solved by using the Corollary \ref{cor_ulin_eqn}.

\end{document}